\documentclass[11pt]{amsart}

\usepackage{graphicx}
\usepackage{gensymb}
\usepackage{pgfplots}
\pgfplotsset{compat=1.18}
\usepackage{textcomp}
\usepackage{amsfonts}
\usepackage{appendix}
\usepackage{listings}
\usepackage{amsmath}
\usepackage{amssymb}
\usepackage{algpseudocode}
\usepackage{algorithm}
\usepackage{amsthm}
\usepackage{amsaddr}
\usepackage{subfig}
\usepackage{booktabs}
\usepackage[hyphens]{url}
\usepackage{multirow}
\usepackage{amsfonts}
\usepackage{amsthm}
\usepackage{mathrsfs}
\usepackage{enumitem}

\newtheorem{theorem}{Theorem}[section]
\newtheorem{proposition}[theorem]{Proposition}
\newtheorem{corollary}[theorem]{Corollary}
\newtheorem{lemma}[theorem]{Lemma}
\newtheorem{conjecture}[theorem]{Conjecture}

\theoremstyle{definition}
\newtheorem{definition}[theorem]{Definition}

\theoremstyle{remark}
\newtheorem*{rem}{Remark}

\def\NN{\ensuremath{\mathbb{N}}}
\def\ZZ{\ensuremath{\mathbb{Z}}}

\def\cA{\ensuremath{\mathcal{A}}}
\def\Pr{\ensuremath{\mathrm{Pr}}}
\def\Var{\ensuremath{\mathrm{Var}}}
\def\Cov{\ensuremath{\mathrm{Cov}}}
\def\EE{\ensuremath{\mathrm{E}}}
\def\nsg{\ensuremath{\mathcal{S}}}

\DeclareMathOperator\F{F}
\DeclareMathOperator\g{g}
\DeclareMathOperator\e{e}
\DeclareMathOperator\Ap{Ap}

\newcommand{\mathdefault}[1][]{}

\begin{document}

\title[Improved upper bounds on key invariants]{Improved upper bounds on key invariants of Erd\H{o}s-R\'enyi numerical semigroups}

\author[Bogart]{Tristram Bogart}
\address{Departamento de Matem\'aticas \\ Universidad de los Andes \\ Bogot\'a, Colombia}
\email{tc.bogart22@uniandes.edu.co}

\author[Morales]{Santiago Morales}
\address{Graduate Group in Applied Mathematics \\ University of California, Davis  \\ Davis, California, USA}
\email{moralesduarte@ucdavis.edu}

\begin{abstract}
  De Loera, O'Neill and Wilburne introduced a general model for random numerical semigroups in which each positive integer is chosen independently with some probability $p$ to be a generator, and proved upper and lower bounds on the expected Frobenius number and expected embedding dimensions. We use a range of probabilistic methods to improve the upper bounds to within a polylogarithmic factor of the lower bounds in each case. As one of the tools to do this, we prove that for any prime $q$, if $\cA$ is a random subset of the cyclic group $\ZZ_q$ whose size is of order $\log(q)$ and $k$ is also of order $\log(q)$, then with high probability the $k$-fold sumset $k\cA$ is all of $\ZZ_q$.

\end{abstract}

\maketitle

\section{Introduction}
A \textit{numerical semigroup} is a subset $\nsg \subseteq \NN$ that is closed under addition, contains 0, and is cofinite. An element of $\NN \setminus \nsg$ is called a \emph{gap} of $\nsg$. Given a finite set $\cA \subseteq \NN$ such that $\gcd(\cA) = 1$, we can define a numerical semigroup
\[ \langle \cA \rangle = \{c_1a_1 + \cdots + c_na_n : c_1, \ldots, c_n \in \NN\}.\]
Conversely, every numerical semigroup has a unique minimal generating set.

For these and many other definitions and properties relating to numerical semigroups, we refer the reader to Rosales and Garc\'ia-S\'anchez's book \cite{rosales2009numerical}. Numerical semigroups arise in toric geometry since their semigroup algebras are exactly the coordinate rings of affine monomial curves. For a range of applications of numerical semigroups, see Assi, D'Anna, and Garc\'ia-S\'anchez's book \cite{assi2020numerical}.

Three of the most fundamental parameters of a numerical semigroup are:
\begin{itemize}
\item the \emph{embedding dimension} $\e(\nsg)$ which is the size of the minimal generating set,
\item the \emph{Frobenius number} $\F(\nsg)$ which is the largest gap, and
\item the \emph{genus} $\g(\nsg)$ which is the number of gaps. 
\end{itemize}
A central problem in numerical semigroups is to prove or disprove \emph{Wilf's conjecture} \cite{wilf, delgado2020conjecture} which is that every numerical semigroup $\nsg$ satisfies the inequality
\[ \frac{\F[\nsg] + 1 - \g(\nsg)}{\F[\nsg]+1} \geq \frac{1}{\e(\nsg)} \]
involving all three parameters. That is, the reciprocal of the embedding dimension is conjectured to be a lower bound on the density of the segment of $\nsg$ between 0 and the Frobenius number. 

As with other discrete structures, it is natural to ask how the key parameters behave in a \emph{random} numerical semigroup. The most extensively studied model of random numerical semigroup is given as follows.  
\begin{definition} \label{def:randnumsems:boxmodel}
Fix $n \in \NN$. For each $T \in \NN$, let \[G(T) = \{A\in \NN^n : \mathrm{gcd}(A) = 1, \; |A|_\infty \leq T\}.\]
A \emph{box model random numerical semigroup} is the uniform probability space over the set of semigroups $\nsg = \langle\cA\rangle$ with $\cA \in G(T)$.
\end{definition}

V. I. Arnold was the first to study this model and he conjectured in \cite{arnold1999weak} and \cite{arnold2004arnold} 
\[\F(\langle\cA\rangle) \sim (n - 1)!^{\frac{1}{n - 1}}(a_1\cdots a_n)^{\frac{1}{n - 1}}.\]
Aliev and Gruber \cite{aliev2007optimal} showed this to be a correct asymptotic lower bound and Aliev, Henk and Hindrichs \cite{aliev2011expected} completed the proof of the conjecture by showing it to also be an asymptotic upper bound. 

Another model of random numerical semigroups is to fix $F \in \NN$ and choose $\nsg$ uniformly among the set of all semigroups whose Frobenius number is exactly $F$. Backelin \cite{Back} showed that as $F$ tends to infinity, almost all such semigroups have no generators less than $\frac{F}{2}-c$ for $c$ constant, so the total number of such semigroups is $O\left(2^{\frac{F}{2}}\right)$. However, if the smallest generator $m$ is close to but does not exceed $\frac{F}{2}$, then there are already $2^m \approx 2^{\frac{F}{2}}$ choices for the generating set; adjoin to $m$ any subset of the numbers between $m+1$ and $2m-1$. It follows that $E[\e(\nsg)]$ is close to $\frac{F}{4}$ and $E[\g(\nsg)]$ is close to $\frac{3F}{4}$. 

A limitation of these models is that each of them fixes some of the parameters while others are allowed to vary. In the first case the embedding dimension is bounded by a constant $n$ while the actual generators, the genus, and the Frobenius number tend to $\infty$. In the second case the Frobenius number is fixed while the embedding dimension and genus are unconstrained except by $F$ itself. In order to study more general spaces of numerical semigroups in which all of the parameters are allowed to vary, De Loera, O'Neill, and Wilburne \cite{de2018random} defined the following model, inspired by the Erd\H{o}s-R\'enyi model of random graphs.

\begin{definition}\label{def:ermodel}
    For $p \in (0, 1)$ and $M \in \NN$, an \emph{ER-type random numerical semigroup} $\nsg(M, p)$ is a probability space over the set of semigroups $\nsg = \langle\cA\rangle$ with $\cA \subseteq \{1,...,M\}$, determined by $\Pr[n \in \cA] = p$ for each $n = 1,\dots, M$ with these events mutually independent.
\end{definition}
We note that this definition does not require a numerical semigroup $\nsg$ to be cofinite, though it is very likely to be so if the expected number of selected generators $Mp$ is large. This follows because $k$ large random numbers are coprime with limiting probability $\frac{1}{\zeta(k)}$, as shown by Nymann \cite{Nymann}.

Another advantage of this model is that a sample of $\nsg(M, p)$ can be obtained by using the following explicit procedure.
\begin{enumerate}
    \item Initialize an empty set $\cA$.
    \item As $i$ goes from $1$ to $M$, add $i$ to $\cA$ with probability $p$, independently of the other steps.
    \item Return the semigroup $\nsg = \langle\cA\rangle$.
\end{enumerate}

\begin{theorem} \cite{de2018random} Let $\nsg \sim \nsg(M, p)$, where $p = p(M)$ is a monotone decreasing function of $M$. Then, 
    \begin{enumerate}[label=(\alph*)]
        \item If $p(M) \in o\left(\frac{1}{M}\right)$, then $\nsg = \{0\}$ almost always.
        \item If $\frac{1}{M} \in o(p(M))$ and $\lim_{M \to \infty} p(M)= 0$, then $\nsg$ is cofinite almost always and
         \[\lim_{M \to  \infty} \EE[\e(\nsg)] = \lim_{M \to \infty} \EE[\g(\nsg)] = \lim_{M \to \infty} \EE[\F(\nsg)] = \infty.\]
        \item If $\lim_{M \to \infty} p(M) > 0$, then
        \[\lim_{M \to \infty} \EE[\e(\nsg)] < \infty,  \quad \lim_{M \to \infty} \EE[\g(\nsg)] < \infty \quad  \text{ and } \quad \lim_{M \to \infty} \EE[\F(\nsg)]< \infty .\]
    \end{enumerate}
\end{theorem}

Note that this proves that $p(M) = \frac{1}{M}$ is a threshold function for cofiniteness.

The proof of part (a) follows from standard arguments using the probabilistic method. On the other hand, parts (b) and (c) follow from the construction of a shellable simplicial complex whose facets are in bijection with \emph{irreducible} numerical semigroups of a fixed Frobenius number $n$. They show that the faces of the $n-$th simplicial complex count the number of sets $\cA \subset \{1, \ldots, n - 1\}$ satisfying $n \notin \langle\cA \rangle$. Thus, the expected value of the embedding dimension can be calculated from the entries of the $h$-vector of this simplicial complex. 

The simplicial complex is also used to prove explicit bounds for the expected value of the embedding dimension and the genus in the case when $p$ is constant.
\begin{theorem} \cite{de2018random} \label{thm:ermodel:constantp}
    Let $\nsg \sim \nsg(M, p)$, where $p$ is a constant. Then, 
    \begin{align*}
        \frac{6 - 8p + 3p^2}{2 - 2p^2 + p^3} &\leq \lim_{M \to \infty} \EE[\e(\nsg)] \leq \frac{2 - p^2}{p}, \\
        \frac{6 - 14p + 11p^2 - 3p^3}{2p - 2p^3 + p^4} &\leq \lim_{M \to \infty} \EE[\g(\nsg)] \leq \frac{(1 - p)(2 - p^2)}{p^2}, \text{ and}\\
        \frac{6 - 14p + 11p^2 - 3p^3}{2p - 2p^3 + p^4} &\leq \lim_{M \to \infty} \EE[\F(\nsg)] \leq \frac{2(1 - p)(2 - p^2)}{p^2}.
    \end{align*}
\end{theorem}

In particular, as $p$ tends to 0 the asymptotics of these bounds, expressed in terms of $\frac{1}{p}$, are as follows.
\begin{itemize}
\item The expected embedding dimension is bounded between a function that tends to a constant and one that is asymptotically linear. 
\item The expected Frobenius number and expected genus are bounded between functions that are asymptotically linear and functions that are asymptotically quadratic. 
\end{itemize}

Our main results consist of asymptotic improvements to these upper bounds that bring them within polylogarithmic factors of the respective lower bounds. Since we are interested in asymptotics in $p$, we consider a modified model in which there is no cutoff $M$ for the size of the selected generators.

\begin{definition} 
    For $p \in (0, 1)$, an \emph{unconstrained ER-type random numerical semigroup} $\nsg(p)$ is a probability space over the set of semigroups $\nsg = \langle\cA\rangle$ with $\cA \subseteq \NN$, determined by $\Pr[n \in \cA] = p$ for each $n$ with these events mutually independent.
\end{definition}

We can choose a semigroup $\nsg$ in the new model and then modify it by adding everything larger than a given number $M$. By letting $M$ tend to infinity, we see that 
\[ \lim_{M \to \infty} \EE[\F(\nsg(M,p))] = \EE[\F(\nsg(p)] \]
and similarly for the embedding dimension. Thus our model allows us to prove results directly comparable with those of De Loera, O'Neill, and Wilburne.

\begin{theorem} \label{thm:main}
As $p$ tends to zero,  
    \begin{align*}
      \EE[\e(\nsg(p))] = O\left( \left( \ln \left( \frac{1}{p} \right) \right)^3 \right) \text{and}  \\
       \EE[\g(\nsg(p))] \leq \EE[\F(\nsg(p))] = O \left( \frac{1}{p} \left( \ln \left( \frac{1}{p} \right) \right)^3 \right).
    \end{align*}
\end{theorem}

In particular, for both the Frobenius number and the embedding dimension, we reduce the asympototic ratio between the lower and upper bounds from order $1/p$ to order $\left(\log(1/p)\right)^3$. The same holds for the genus because it is trivially bounded by the Frobenius number.  

Our proofs will be purely probabilistic in the sense of not introducing auxiliary objects such as the simplicial complexes mentioned above. The difficulty in applying probabilistic methods to numerical semigroups (as compared, for example, to graphs or hypergraphs) is that elements of a random numerical semigroup $\nsg$ are created via two simultaneous processes, one probabilistic and the other algebraic. The first is the random selection of generators of $\nsg$ and the second is the creation of new elements of $\nsg$ as sums of the generators. 

Our proof strategy is to separate the two processes as follows. First, we show that with high probability, both a prime $q$ of an appropriate size (depending on $p$) and a set $\cA$ of roughly $\log(q)$ elements less than $q$ are selected as generators of $\nsg$. The set $\cA$ can be interpreted as a random subset of the cyclic group $\ZZ_q$, and we prove a result (Theorem \ref{thm:sumsets}) about $k$-fold sums of such subsets that may be of independent interest in additive number theory. We use this result to show that with high probability, the Frobenius number is bounded by roughly $\log^2(q)$.  We use simpler methods to bound the conditional expectation of the Frobenius number in the case that the high-probabilility events described above fail to occur. Finally, we adapt most of the same ideas to prove the bound on the expected embedding dimension.

  
The organization of the rest of the paper is as follows. In Section 2 we prove Theorem \ref{thm:sumsets} via the second moment method. In Section 3 we apply Theorem \ref{thm:sumsets} to show that in the model $\nsg(p)$, the Frobenius number is bounded by an appropriate function of $\frac{1}{p}$ with high probability. In Section 4 we complete the proof of Theorem \ref{thm:main}. In Section 5 we present the results of experiments that suggest that the true asymptotics of the expected Frobenius number and expected embedding dimension lie strictly between the lower bounds and our new upper bounds. More precisely, in both cases, the experiments point to exactly one factor of $\log(1/p)$ instead of zero (as in the lower bounds) or three (as in the upper bounds.) We end with some suggestions about how the upper bounds might be further improved.


\section{Sumsets of Random Subsets of $\ZZ_q$}

Let $\cA$ be a subset of an abelian group $G$. The \emph{$k$-fold sumset} is defined as
\[ k \cA = \{ a_1 + \dots + a_k \, | a_1, \dots, a_k \in \cA \}. \]
Sumsets are a central object of study in additive number theory (see for example Tao and Vu's book \cite{TaoVu}.) However, there are fewer results on sumsets of \emph{randomly chosen} sets $\cA \subset G$. Lee and Oh \cite{LeeOh} recently estimated the number of \emph{zero-free subsequences} of sumsets in this context. In this section, we will prove that if a random subset $\cA \subset \ZZ_q$ as well as $k$ are at least a large multiple of $\log(q)$, then with high probability $k \cA$ covers the entire group $\ZZ_q$.

\begin{theorem} \label{thm:sumsets}
  Let $q$ be a prime number and $\cA$ be a random subset of $\ZZ_q$ of size $2 \lceil b \log_2 q \rceil$. Then 
  \[ \Pr[ \left( b \log_2 q \right) \cA \neq \ZZ_q] \leq \frac{2b \log_2 q + 3}{q^{b-2}}.\]

\end{theorem}

\begin{rem}
  \begin{enumerate}
    \item Our proof uses combinatorial bounds on binomial coefficients in order that the inequalities hold for \emph{all} $q$. As $q$ tends to $\infty$ one could instead use Stirling's formula to obtain the improved asymptotic bound
  \[ \Pr \left( \lceil b \log_2 q \rceil \cA \neq \ZZ_q \right) \in  O \left( \frac{\ln q}{q^{2b-2}} \right). \]

\item Our proof actually bounds the probability that not every element of $\ZZ_q$ is a sum of  $b \log_2 q$ \emph{distinct} elements of $\cA$, so it yields a slightly stronger statement.
 \end{enumerate}
\end{rem}

To prove this theorem we will apply the second moment method for discrete random variables, as explained in Alon and Spencer's book \cite[\S 4.3]{AlSp}. If $X$ is any $\NN$-valued random variable then it follows from Chebyshev's Inequality that
\[ \Pr[X > 0] \leq \frac{\Var[X]}{\EE[X]^2}.\]
Suppose $X = \sum_{i=1}^m X_i$ is a sum of indicator variables. Let 
\[ \Delta = \sum_{i_1 \neq i_2} \Cov[X_{i_1}, X_{i_2}] = \sum_{i_1 \neq i_2} \Pr[A_i \land A_j] .\]
By decomposing $\Var[X]$ in terms of covariances it follows that
\begin{equation}
  \Pr[X > 0] \leq \frac{\EE[X]+ \Delta}{\EE[X]^2}. \label{eq:Delta}
\end{equation}

\begin{proof}[Proof of Theorem \ref{thm:sumsets}] Let $k \leq s \leq q$ and let $\cA$ be a uniformly random subset of $\ZZ_q$ of size $s$. For each $z \in \ZZ_q$, let
  \[N_z^k := \left\{K \subseteq \mathbb{Z}_q: |K| = k, \sum_{t \in K} t = z\right\}.\]
  Note that $|N_z^k| = \frac{1}{q}{\binom{q}{k}}$, since $K \in N_0^k$ if and only if $K + k^{-1}z \in N_z^k$ for every $z \in \mathbb{Z}_q$.

  For each $K \in N_z^k$, let $W_K$ be the indicator variable for the event $E_K$ that $K \subset \cA$, and define the random variable
  \[X_z = \sum_{K \in N_z^k} W_K\]
  which counts the number of subsets of $\cA \cap N_z^k$.  Since the sum of every $K \subset \cA$ is in $\mathbb{Z}_q$,
\[\sum_{z \in Z_q} X_z = {\binom{s}{k}},\]
and so
\[\binom{s}{k} = \EE\left[\sum_{z \in Z_q} X_z\right] =  \sum_{z \in Z_q} \EE[X_z].\]\par
As in the calculation of $|N_z^k|$, for every $z \in \mathbb{Z}_q$, 
\[\EE[X_0] = \sum_{K \in N_0^k}\EE[W_K] = \sum_{K \in N_0^k} \EE[W_{K + k^{-1}z}] = \sum_{K \in N_z^k} \EE[W_K] = \EE[X_z].\]
Therefore, we have that
\begin{equation} \label{eq:upperbound:expected}
\EE[X_z] = \frac{1}{q} {\binom{s}{k}}.
\end{equation}

We next need to bound the variance of $X_z$. For each $0 \leq j \leq k$, let
\[\Delta_j := \sum_{\substack{|K \cap L| = j \\ L \in N_z^K}} \Pr[E_K \land E_L]\]
so that $\Delta = \sum_{j=0}^k \Delta_j$. We can bound the number of events for which $|K \cap L| = j$. First we choose $K$ as any set in $N_z^k$ and then we choose the remaining $k- j$ elements of $L$ as any subset of $\mathbb{Z}_q \setminus K$ with size $k - j$ (so that the choices of $L$ include every set in $N_z^k$ as well as others.) Thus, 
\[\Delta_j \leq \frac{1}{q}\binom{q}{k}\binom{q - k}{k - j} \frac{\binom{q - 2k + j}{s - 2k + j}}{\binom{q}{s}}.\]

This implies that, using (\ref{eq:upperbound:expected}),
\begin{align*}
    \frac{\Delta_j}{\EE[X_z]^2} &\leq \frac{\binom{q}{k} \binom{q - k}{k - j}\binom{q - 2k + j}{s - 2k + j}}{\frac{1}{q} \binom{s}{k}\frac{1}{q} \binom{s}{k}q\binom{q}{s}} \\
    &= \frac{\frac{q!}{(q - k)!k!}\frac{(q - k)!}{(k - j)!(q - 2k + k)!}\frac{(q - 2k + j)!}{(s - 2k + j)!(q - s)!}}{\frac{1}{q}\binom{s}{k}\frac{s!}{(s - k)!k!}\frac{q!}{(q - s)!s!}} \\
    &= \frac{q\binom{s - k}{k - j}}{\binom{s}{k}}.
\end{align*}

In particular, if $k$ is even and $s=2k$, then $\binom{s - k}{k - j} = \binom{k}{k-j}$ is maximized when $j=k/2$. Thus
\[ \frac{\Delta_j}{\EE[X_z]^2} \leq \frac{\binom{k}{k/2}}{q \binom{2k}{k}} \leq \frac{\binom{k}{k/2}}{\binom{k}{k/2}\binom{k}{k/2}} \leq \frac{q}{2^{\frac{k}{2}}}.\]

  For $k=2 \lceil b\log_2 q \rceil$, this yields
  \[ \frac{\Delta_j}{\EE[X_z]^2} \leq \frac{q}{2^{\lceil b \log_2 q \rceil}} \leq \frac{q}{2^{b \log_2 q }} = \frac{q}{q^{b}} = \frac{1}{q^{b-1}}.\]
  For the same value of $k$ we also have
\[ \EE[X_z] = \frac{1}{q} \binom{2k}{k} > \frac{2^k}{q} \geq \frac{q^{2b}}{q}=q^{2b-1}.\] 
  Thus, using (\ref{eq:Delta}) we obtain
  \begin{align*} \Pr[X_z = 0] & \leq \frac{\EE[X_z]+\Delta}{\EE[X_z]^2} \\
     = \frac{1}{\EE[X_z]} + \sum_{j=0}^k\frac{\Delta_j}{\EE[X_z]^2} & \leq \frac{1}{q^{2b-1}}+\frac{k+1}{q^{b-1}} \\
      \leq \frac{k+2}{q^{b-1}} & = \frac{2b \log_2 q + 3}{q^{b-1}}
    \end{align*}
 
  and so by the union bound,
  \[ \Pr \left[ \bigvee_{z \in \ZZ_q} X_z = 0 \right] \leq q \left( \frac{2b \log_2 q + 3}{q^{b-1}}\right) = \frac{2b \log_2 q + 3}{q^{b-2}}. \qedhere \]
\end{proof}


\section{Ap\'ery sets and a typical generation process}
In this section we will describe a series of events that imply a bound on the expected Frobenius number and expected dimension of a numerical semigroup chosen via the model $\nsg(p)$, and then bound the probability that any of these events fail to occur. The key to this bound will be Theorem \ref{thm:sumsets}.

Let $\nsg$ be a numerical semigroup. For each $m \in \nsg \setminus \{0\}$, the \emph{Ap\'ery set} of $m$ in $\nsg$ is defined to be
\[ \Ap(\nsg, m) = \{x \in \nsg \, | \, x - m \notin \nsg \}.\]
For each $i = 0,1,2\dots,m-1$, the largest gap $y$ of $\nsg$ such that $y \equiv i$ (mod $m$), if there is any such gap, equals $x - m$ where $x$ is the unique element of $\Ap(\nsg, m)$ such that $y \equiv i$ (mod $m$). It follows that for any $m \in \nsg \setminus \{0\}$,
\begin{equation} \F(\nsg) = \max \{ \Ap(\nsg, m) \} - m < \max \{ \Ap(\nsg, m) \}. \label{eq:Apery} \end{equation}



Define
\[ f(p) = \frac{1}{p} \left( \ln \left( \frac{1}{p} \right) \right)^2 \]
and consider the following events for a semigroup $\nsg$ chosen via the random model $\nsg(p)$.
 \begin{itemize}
\item Define $D_1$ to be the event that at least one prime number between $f(p)+1$ and $6f(p)$ is selected as a generator of $\nsg$. 
\item Assuming $D_1$, let $q$ be the largest prime selected in this range. Define $D_2$ to be the event that $D_1$ occurs and also that at least $12 \log_2 q$ generators are selected between 1 and $q-1$. 
\item Assuming $D_2$, let $\nsg'$ be the subsemigroup of $\nsg$ generated by a random subset of $12 \log_2q$ of the generators of $\nsg$ between 1 and $q$, along with $q$ itself. Let $D_3$ be the event that $D_2$ occurs and also that the largest element of the Ap\'ery set of $\nsg'$ with respect to $q$ is at most $6q \log_2 q$.
 \end{itemize}

If $D_3$ holds, then the semigroup $\nsg'$ defined in $D_3$ contains $\nsg$. Using (\ref{eq:Apery}) and the upper bound on $q$, we can bound the Frobenius number of $\nsg$ by 
\begin{equation} \F(\nsg) \leq \F(\nsg') \leq 6q \log_2 q \leq 36f(p) \log_2(6f(p)) \leq K \left(\frac{1}{p} \left( \ln \left(\frac{1}{p}\right)\right)^3  \right). \label{eq:u-function}
\end{equation}
for a sufficiently large constant $K$.

To bound the probability that $D_3$ does not hold, we begin by analyzing $D_1$ via classical number theory.  

\begin{proposition} \label{prop:notD1}
 The probability that $D_1$ does not hold is $o\left(p^4\right)$. 
\end{proposition}

\begin{proof}
  Let $\pi(x)$ denote the number of primes less than or equal to $x$. The prime number theorem states that as $x$ tends to $\infty$, $\pi(x) = \frac{x}{\ln x} \left( 1 + o(1) \right)$. It follows that for any constant $c > 1$,
  \[ \pi(cx) - \pi(x) = \frac{(c-1)x}{\ln x} \left( 1 \pm o(1) \right). \]
  Let $N(p)$ be the number of primes $q$ such that $f(p)+1 \leq q \leq 6f(p)$. Since $f(p)$ tends to infinity as $p$ tends to zero, we have  

\[ N(p) = \frac{\frac{5}{p} \left( \ln \left( \frac{1}{p} \right) \right)^2}
    {\ln \left( \frac{1}{p} \left( \ln \left( \frac{1}{p} \right) \right)^2 \right)}
    \left( 1 \pm o(1) \right) = \frac{\frac{5}{p} \left( \ln \left( \frac{1}{p} \right) \right)^2}
         {\ln \left( \frac{1}{p} \right) \left( 1 + o(1) \right)}
         \left( 1 \pm o(1) \right) \\
        = \frac{5}{p} \ln \left( \frac{1}{p} \right) \left( 1 \pm o(1) \right).\]

  It follows that
  \[ \Pr\left[ \lnot D_1 \right] = (1-p)^{N(p)} \leq e^{N(p)} = e^{\frac{5}{p} \ln \left( \frac{1}{p} \right) \left( 1 \pm o(1) \right)} 
    = p^{5 \left( 1 \pm o(1) \right)}  = o\left(p^4\right).\]
\end{proof}

\begin{rem}
  We could make the bound on $\Pr(\lnot D_1)$ effective by applying an effective version of the prime number theorem, such as Rosser's classical result \cite{rosser1941} that for all $x \geq 55$,
  \[ \frac{x}{\ln x + 2} < \pi(x) < \frac{x}{\ln x - 4}.\]
\end{rem}
\quad

Next, to analyze the event $D_2$, assume that $D_1$ holds and define $q$ to be the largest prime between $f(p)$ and $cf(p)$ that is selected as a generator. Let $G$ be the number of selected generators between 1 and $q-1$. 

\begin{lemma} \label{lem:expectedG}
$\EE[G] = (q-1)p$. 
\end{lemma}

\begin{rem}
The reason that this may not be obvious is that $q$ itself is a random variable that depends on which generators were selected. Indeed, if we had defined $q$ to be the smallest prime or a random prime rather than the \emph{largest} prime among the generators between $f(p)+1$ and $cf(p)$, then this statement would not hold.     
\end{rem}

\begin{proof}
  Write $G = \sum_{\ell=1}^{q-1} G_\ell$ where $G_\ell$ is an indicator variable for $\ell$ being selected as a generator. If $\ell < f(p)$ or if $\ell$ is not prime, then it is clear that $G_\ell$ is independent of the choice of $q$, so $\EE[G_{\ell}]=p$. For the remaining cases, let $q_1 < q_2 < \dots < q_n$ be the primes between $f(p)$ and $cf(p)$. For any $1 \leq i < j \leq n$,
  \[ \EE[G_{q_i} | q = q_j] = \frac{\Pr\left[ (\textup{select }q_i) \land (q = q_j)\right]}{\Pr[q = q_j]} = \frac{p^2(1-p)^{n-j}}{p(1-p)^{n-k}} = p,\]
  so the result follows by linearity of expectation. 
\end{proof}

\begin{proposition} \label{prop:notD2} The conditional probability $\Pr[\lnot D_2 | D_1]$ is $O\left(p^{\frac{1}{8} \ln \left( \frac{1}{p} \right)}\right)$. In particular, it decays faster than any fixed power of $p$. 
\end{proposition}

\begin{proof}
As in Lemma \ref{lem:expectedG}, let $G$ be the number of selected generators between $1$ and $q-1$. By this lemma we have 
\[\EE[G] = p(q-1)  \geq pf(p) = \left( \ln \left( \frac{1}{p} \right) \right)^2.\]
On the other hand,
\[ 12 \ln q \leq 12 \ln \left( \frac{5}{p} \left( \ln \left( \frac{1}{p} \right) \right)^2 \right) = O \left(\ln \left( \frac{1}{p} \right)  \right).\]
Thus for sufficiently small $p$, we have $12 \ln q < \frac{\EE[G]}{2}$. 
We apply the Chernoff bound
\[        \Pr[G \leq \EE[G] - \lambda] \leq e^{-\frac{\lambda^2}{2\EE[G]}} \]
with $\lambda = \frac{\EE[G]}{2}$ to obtain
\[ \Pr[\lnot D_2 | D_1] \leq \Pr\left[G \leq \frac{\EE[G]}{2}\right] \leq e^{-\frac{\EE[G]}{8}} \leq e^{\frac{-1}{8} \left( \ln \left( \frac{1}{p} \right) \right)^2} = p^{\frac{1}{8} \ln \left( \frac{1}{p} \right)}. \qedhere \]
\end{proof}

\begin{rem}
To get a significantly faster decay in Proposition \ref{prop:notD2} in practice, we could select an appropriate constant $d$ and modify $f(p)$ by replacing both occurrences of $\frac{1}{p}$ by $\frac{d}{p}$. The constant $\frac{1}{8}$ in the exponent would become $\frac{d}{8}$. The tradeoff is that the bound on the Frobenius number in (\ref{eq:u-function}) becomes somewhat weaker. We could also apply Chernoff more tightly for appropriately small values of $p$. 
\end{rem}

We now assume that events $D_1$ and $D_2$ both hold, so that we can use Theorem \ref{thm:sumsets} in order to analyze $D_3$.

\begin{proposition} \label{prop:notD3}
  The conditional probability $Pr[\lnot D_3 | D_1 \land D_2]$ is $o\left(p^4\right)$. 
\end{proposition}

\begin{proof}
  By applying Theorem \ref{thm:sumsets} with $b=6$ along with the lower and upper bounds on the prime $q$, we have
  \[ \Pr[ \lnot D_3 | D_1 \land D_2] \leq \frac{12 \log_2 q + 3}{q^4}
  \leq \frac{12 \log_2 \left( \frac{6}{p} \left( \ln \left( \frac{1}{p} \right) \right)^2 \right) + 3 }{\left( \frac{1}{p} \left( \ln \left( \frac{1}{p} \right) \right)^2 \right)^4 }
       = O \left( \frac{\ln \left( \frac{1}{p} \right)}{\left( \frac{1}{p} \right)^4 \left( \ln \left( \frac{1}{p} \right) \right)^8 } \right) 
       \]
       which is $o\left(p^4\right)$. 
\end{proof}

\begin{corollary} \label{cor:notD3}
  The probability $\Pr[\lnot D_3]$ that the Ap\'ery set process fails is $o\left(p^4\right)$. 
\end{corollary}

\begin{proof}
 We calculate 
 \begin{align*} \Pr[\lnot D_3] & = \Pr[\lnot D_1] + \Pr[D_1 \land \lnot D_2] + \Pr[D_1 \land D_2 \land \lnot D_3] \\
   & \leq \Pr[\lnot D_1] + \Pr[\lnot D_2 | D_1] + \Pr[\lnot D_3 | D_1 \land D_2] \end{align*}
  and all three terms are $o\left(p^4\right)$ by Propositions \ref{prop:notD1}, \ref{prop:notD2}, and \ref{prop:notD3}.
\end{proof}


\section{Proof of the Main Theorem}
To prove Theorem \ref{thm:main}, we will combine Corollary \ref{cor:notD3} with an alternative bound on the Frobenius number when the event $D_3$ does not hold, as follows. It is well known (\cite[Proposition 2.13]{rosales2009numerical}) that for any coprime numbers $n_1, n_2$, the Frobenius number of the semigroup $\langle n_1, n_2 \rangle$ is exactly $(n_1-1)(n_2-1)-1$. In particular, as soon as two consecutive numbers $2n, 2n+1$ (the first one even) are selected as generators of $\nsg$, then
\begin{equation} \F(\nsg) \leq \F\left( \langle 2n, 2n+1 \rangle \right) = (2n-1)(2n) - 1 < 4n^2. \end{equation}

\begin{lemma} \label{lem:unlikelycase}
  Given any $u = u(p)$,
  \[ \EE[\F(\nsg) | \F(\nsg) \geq u] = \frac{8}{p^4} + \frac{4u}{p^2} + u^2 .\]
\end{lemma}
\begin{proof}
  Let $L$ be the smallest number such that $L > u/2$ and such that both $2L$ and $2L+1$ are selected as generators of $\nsg \in \nsg(p)$. For any given $p$, this random variable is finite with probability one. Also, $L$ depends only on the selection of generators greater than $u$. On the other hand, the event that $\F(\nsg) \geq u$ depends only on the generators less than or equal to $u$, so $L$ is independent of this event. Thus
  \begin{align*}
    \EE[\F(\nsg) | \F(\nsg) \geq u] & \leq \EE[\F\left( \langle 2L, 2L+1 \rangle \right)] \\
   \leq \EE[4L^2] & = \sum_{n=\frac{u}{2}+1}^{\infty}4n^2\left(1-p^2\right)^{n-\frac{u}{2}}p^2 \\
    & = 4p^2 \sum_{m=1}^{\infty}\left( m + \frac{u}{2}\right)^2\left(1-p^2\right)^m \\
    & \leq 4p^2 \sum_{m=1}^{\infty}\left( m + \frac{u}{2}\right)^2 e^{-p^2m} \\
    & \leq 4p^2 \int_0^\infty \left( x + \frac{u}{2}\right)^2 e^{-p^2x} \, dx \\
    & = 4p^2 \left( \frac{8+4p^2u+p^4u^2}{4p^6} \right) \\
    & = \frac{8}{p^4} + \frac{4u}{p^2} + u^2. \qedhere
  \end{align*}  
\end{proof}

Let $F$ and $Z$ be random variables that represent the Frobenius number and the embedding dimension, respectively, of a numerical semigroup $\nsg \in \nsg(p)$. As we have seen,
we can bound the probability of the undesirable outcome $\lnot D_3$ by $o\left(p^4\right)$ to cancel out the factor of $O \left( \frac{1}{p^4}\right)$ in Lemma \ref{lem:unlikelycase}. Choose \[u(p) = K \left(\frac{1}{p} \left( \ln \left(\frac{1}{p}\right)\right)^3  \right)\]
  where the constant $K$ is as in (\ref{eq:u-function}), so that $\F(\nsg) \leq u(p)$ whenever the event $D_3$ holds. 


  \begin{proof}[Proof of Theorem \ref{thm:main}] To prove the upper bounds on both the expected Frobenius number and the expected embedding dimension, we will condition on whether or not the Frobenius number is bounded by $u(p)$.

    For the expected Frobenius number, we apply Corollary \ref{cor:notD3} along with (\ref{eq:u-function}) to obtain
\begin{align*}
  \EE[F] & = \Pr[F < u(p)]\EE[F | F < u(p)] + \Pr[F \geq u(p)]\EE[F | F \geq u(p)] \\
  & \leq 1 \cdot u(p) + \Pr[\lnot D_3] \EE[F | F \geq u(p)] \\
  & = O\left(\frac{1}{p} \left( \ln \left(\frac{1}{p}\right)\right)^3 \right) + o \left(p^4\right)O\left(\frac{1}{p^4}\right) = O\left(\frac{1}{p} \left( \ln \left(\frac{1}{p}\right)\right)^3 \right)
\end{align*}

We now turn to the expected embedding dimension. For any semigroup $\nsg$ with Frobenius number $f$, by definition the numbers $f+1, f+2, \dots, 2f+1$ all belong to $\nsg$ and these numbers alone generate every number greater than $2f+1$. Thus no minimal generator is larger than $2f+1$, and also $f$ itself is not a generator, so a loose bound on the embedding dimension is $\e(\nsg) \leq 2\F(\nsg)$. That is, our random variables satisfy $Z \leq 2F$. We will apply this bound in the case that $F \geq u(p)$, and thus reuse the calculation of the second term above. 

To handle the case where $F < u(p)$, let $Y$ denote the number of generators selected up to $2u(p)$. This is just a binomial random variable. If $F < u(p)$ then by the same argument as above, no minimal generator is greater than $2u(p)$. That is, $Z \leq Y$ in this case. We calculate
\begin{align*} \EE[Z] & = \Pr[F < u(p)]\EE[Z | F < u(p)] + \Pr[F \geq u(p)]\EE[Z | F \geq u(p)] \\
  & \leq 1 \cdot \EE[Y] +  \Pr[F \geq u(p)]\EE[2F | F \geq u(p)] \\
  & = O(pu(p)) + o(1) = O\left(\left( \ln \left(\frac{1}{p}\right)\right)^3 \right).\qedhere
\end{align*}
\end{proof}


\section{Experiments, Conclusions, and Future Work}
By Theorems \ref{thm:ermodel:constantp} and \ref{thm:main} we have seen that for sufficiently small $p$, the expected Frobenius number of the  Erd\H{o}s-R\'enyi numerical semigroup $\nsg(p)$ satisfies
\[ \frac{K_1}{p} \leq \EE[\F(\nsg(p))] \leq \frac{K_2}{p} \left( \log_2 \left(\frac{1}{p}\right)\right)^3 \]
and that the expected embedding dimension satisfies
\[ K_3 \leq \EE[\e(\nsg(M,p))] \leq K_4 \left( \log_2 \left(\frac{1}{p}\right)\right)^3 \] 
for appropriate constants $K_1, \dots, K_4$. However, extensive experiments \cite{Morales_data}, in which 1000 Erd\H{o}s-R\'enyi semigroups were generated for each of fifteen values of $p$, suggest the following conjecture.
\begin{conjecture}
  \begin{enumerate}
  \item The expected Frobenius number $\EE[\F(\nsg(p))]$ is of order $\frac{1}{p} \log \left(\frac{1}{p}\right)$.
  \item The expected embedding dimension $\EE[\e(\nsg(p))]$ is of order $\log \left(\frac{1}{p}\right)$. 
  \end{enumerate}
  \end{conjecture}
That is, it appears that both of our new upper bounds are too large by two log factors while the lower bounds of De Loera, O'Neill, and Wilburne are too small by a single log factor. Furthermore, as we show in Figures \ref{fig:frobenius} and \ref{fig:embedding}, the true constant factors appear to be quite small. Additionally, this asymptotic behavior appears to apply even when $p$ is relatively large. 
\begin{figure}[ht]
            \begin{center}
                \scalebox{0.6}{\input{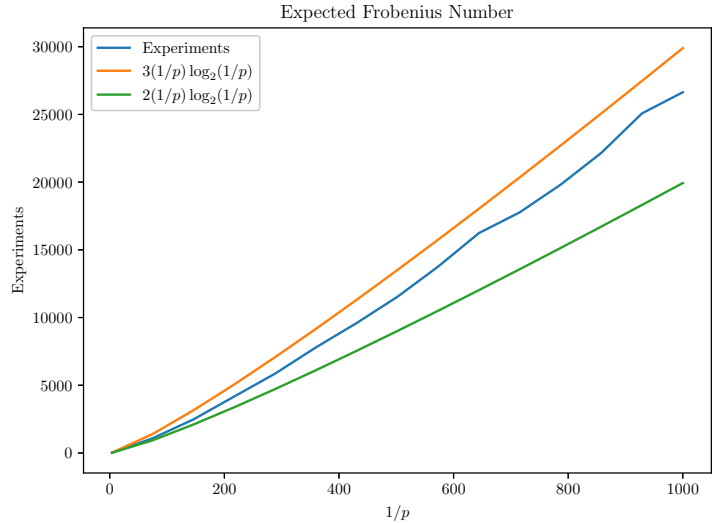}}
            \end{center}
            \caption{Average Frobenius number of random numerical semigroups generated using the ER-type model vs $(1/p)\log_2(1/p)$.}
            \label{fig:frobenius}
\end{figure}

\begin{figure}[ht]
        \begin{center}
            \scalebox{0.6}{\input{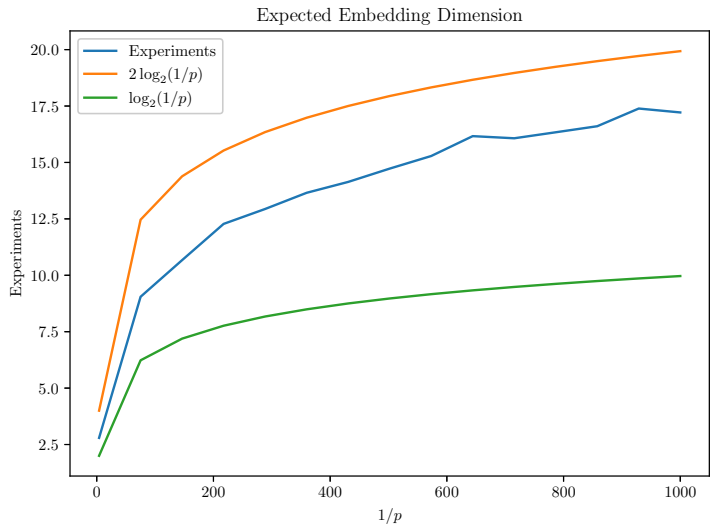}}
        \end{center}
        \caption{Average embedding dimension of random numerical semigroups generated using the ER-type model vs $\log_2(1/p)$.}
        \label{fig:embedding}
\end{figure}

We propose some possible approaches to improve the asymptotic upper bounds to more closely match the data. 
\begin{enumerate}
\item The proof of Theorem \ref{thm:main} involves waiting for a prime generator $q$ to be selected. By the Prime Number Theorem, the first prime generator is expected to be of order roughly $\frac{1}{p} \log \left(\frac{1}{p}\right)$. However, the first generator and in fact the first $k$ generators for any constant $k$ are expected to be of order $\frac{1}{p}$. Assuming the very likely event that (say) $k$ generators $g_1, \dots, g_k$ are selected between 1 and $\frac{2k}{p}$, let $G \subseteq \nsg$ be the $m$-fold sumset of $\{g_1, \dots, g_k\}$ where $m$ is slightly more than $\log \log \left( \frac{1}{p} \right)$. Then $\max\{G\}$ is of order slightly more than $\frac{1}{p} \log \log \left( \frac{1}{p} \right)$ and $|G|$ is of order slightly more than $\log \left( \frac{1}{p} \right)$  If the primes were a truly random subset of $\NN$ of density $\frac{1}{\ln n}$, then $G$ would be expected to contain a prime. We do not know how to estimate the probability that $G$ actually contains a small prime, but if it were possible to do so, then this prime could replace $q$ in the proof of the main theorem and nearly eliminate one of the log factors.      

\item Alternatively, one could ignore primes entirely and simply try to bound the Ap\'ery set with respect to the smallest generator of $\nsg$. This might be feasible via some generalization of Theorem \ref{thm:sumsets} in which $q$ would not be required to be prime. The proof would need to be modified since it relies on the symmetry of prime cyclic groups.
\end{enumerate}

To give an idea of the behavior of Erd\H{o}s-R\'enyi numerical semigroups, we end with two pictures of semigroups chosen from $\nsg(M,p)$ with $M$ large compared to $\frac{1}{p}$. This is effectively equivalent to $\nsg(p)$ because with high probability the Frobenius number is smaller than $M$.

\begin{figure}[ht]
        \begin{center}
            \scalebox{0.6}{\input{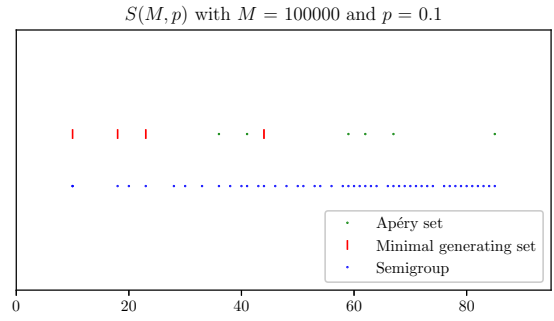}}
        \end{center}
\caption{A numerical semigroup chosen from $\nsg(100000, 0.1)$}     
\end{figure}

\begin{figure}[ht]
        \begin{center}
            \scalebox{0.6}{\input{ertype_visual.pgf}}
        \end{center}
        \caption{A numerical semigroup chosen from $\nsg(100000, 0.01)$}
\end{figure}





\textbf{Acknowledgments:} The authors thank Jes\'us de Loera, Christopher O'Neill, Adolfo Quiroz, and an anonymous referee of a previous draft of this manuscript for their helpful comments. Tristram Bogart was supported by internal research grant INV-2025-213-3438 from the Faculty of Sciences of the Universidad de los Andes.

\bibliography{library}

\providecommand{\bysame}{\leavevmode\hbox to3em{\hrulefill}\thinspace}
\providecommand{\MR}{\relax\ifhmode\unskip\space\fi MR }
\providecommand{\MRhref}[2]{%
  \href{http://www.ams.org/mathscinet-getitem?mr=#1}{#2}
}
\providecommand{\href}[2]{#2}
\begin{thebibliography}{DLOW18}

\bibitem[ADGS20]{assi2020numerical}
Abdallah Assi, Marco D'Anna, and Pedro~A Garc{\'\i}a-S{\'a}nchez,
  \emph{Numerical semigroups and applications}, vol.~3, Springer Nature, 2020.

\bibitem[AG07]{aliev2007optimal}
Iskander~M Aliev and Peter~M Gruber, \emph{An optimal lower bound for the
  {F}robenius problem}, Journal of Number Theory \textbf{123} (2007), no.~1,
  71--79.

\bibitem[AHH11]{aliev2011expected}
Iskander Aliev, Martin Henk, and Aicke Hinrichs, \emph{Expected {F}robenius
  numbers}, Journal of Combinatorial Theory, Series A \textbf{118} (2011),
  no.~2, 525--531.

\bibitem[Arn99]{arnold1999weak}
Vladimir~Igorevich Arnold, \emph{Weak asymptotics for the numbers of solutions
  of {D}iophantine problems}, Functional Analysis and Its Applications
  \textbf{33} (1999), no.~4, 292--293.

\bibitem[Arn04]{arnold2004arnold}
Vladimir~I Arnold, \emph{{A}rnold's problems}, Springer, 2004.

\bibitem[AS16]{AlSp}
Noga Alon and Joel~H Spencer, \emph{The {P}robabilistic {M}ethod}, John Wiley
  \& Sons, 2016.

\bibitem[Bac90]{Back}
J{\"o}rgen Backelin, \emph{On the number of semigroups of natural numbers},
  Mathematica Scandinavica (1990), 197--215.

\bibitem[Del20]{delgado2020conjecture}
Manuel Delgado, \emph{Conjecture of {W}ilf: a survey}, Numerical semigroups,
  Springer INdAM Ser., vol.~40, Springer, Cham, 2020, pp.~39--62.

\bibitem[DLOW18]{de2018random}
Jesus De~Loera, Christopher O'Neill, and Dane Wilburne, \emph{Random numerical
  semigroups and a simplicial complex of irreducible semigroups}, The
  Electronic Journal of Combinatorics (2018), P4--37.

\bibitem[LO23]{LeeOh}
Sang~June Lee and Jun~Seok Oh, \emph{On zero-sum free sequences contained in
  random subsets of finite cyclic groups}, Discrete Applied Mathematics
  \textbf{330} (2023), 118--127.

\bibitem[Mor23]{Morales_data}
Santiago Morales, \emph{randnumsgps},
  \url{https://github.com/smoralesduarte/randnumsgps}, 2023.

\bibitem[Nym72]{Nymann}
James~E Nymann, \emph{On the probability that k positive integers are
  relatively prime}, Journal of Number Theory \textbf{4} (1972), no.~5,
  469--473.

\bibitem[RGS09]{rosales2009numerical}
Jos{\'e}~Carlos Rosales and Pedro~A Garc{\'\i}a-S{\'a}nchez, \emph{Numerical
  semigroups}, Springer, 2009.

\bibitem[Ros41]{rosser1941}
Barkley Rosser, \emph{Explicit bounds for some functions of prime numbers},
  American Journal of Mathematics \textbf{63} (1941), no.~1, 211--232.

\bibitem[TV06]{TaoVu}
Terence Tao and Van~H Vu, \emph{Additive combinatorics}, vol. 105, Cambridge
  University Press, 2006.

\bibitem[Wil78]{wilf}
Herbert~S Wilf, \emph{A circle-of-lights algorithm for the “money-changing
  problem”}, The American Mathematical Monthly \textbf{85} (1978), no.~7,
  562--565.

\end{thebibliography}
\bibliographystyle{amsalpha}

\end{document}